\documentclass{amsart}

\newtheorem{thm}{Theorem}[section]

\newtheorem{prop}{Proposition}[section]
\newtheorem{coro}{Corollary}[section]

\newtheorem{alphthm}{Theorem}[section]

\newcommand{\lct}{\; \raisebox{-.96ex}{$\stackrel{\textstyle <}{\sim}$} \;}

\newcommand{\mbb}{\mathbb}

\begin{document}

\title{Mizohata-Takeuchi estimates in the plane}

\author{Bassam Shayya}
\address{Department of Mathematics\\
         American University of Beirut\\
         Beirut\\
         Lebanon}
\email{bshayya@aub.edu.lb}

\date{August 22, 2022}


\begin{abstract}
Suppose $S$ is a smooth compact hypersurface in $\mbb R^n$ and $\sigma$ is 
an appropriate measure on $S$. If $Ef= \widehat{fd\sigma}$ is the extension 
operator associated with $(S,\sigma)$, then the Mizohata-Takeuchi conjecture 
asserts that $\int |Ef(x)|^2 w(x) dx \lct$
$(\sup_T w(T)) \| f \|_{L^2(\sigma)}^2$ for all functions 
$f \in L^2(\sigma)$ and weights $w : \mbb R^n \to [0,\infty)$, where the 
$\sup$ is taken over all tubes $T$ in $\mbb R^n$ of cross-section 1, and 
$w(T)= \int_T w(x) dx$. This paper investigates how far we can go in proving 
the Mizohata-Takeuchi conjecture in $\mbb R^2$ if we only take the decay 
properties of $\widehat{\sigma}$ into consideration. As a consequence of our 
results, we obtain new estimates for a class of convex curves that include 
exponentially flat ones such as $(t,e^{-1/t^m})$, $0 \leq t \leq c_m$, 
$m \in \mbb N$.
\end{abstract}

\maketitle

\section{Introduction}

Let $S$ be a smooth compact hypersurface in $\mbb R^n$, and $\sigma$ a 
finite Borel measure on $S$. The Fourier extension operator $E$ associated 
with the pair $(S,\sigma)$ is defined as
\begin{displaymath}
Ef(x)
= \widehat{f d\sigma}(x) = \int e^{-2\pi i x \cdot \xi} f(\xi) d\sigma(\xi)
\end{displaymath}
for $f \in L^1(\sigma)$.

Suppose $w$ is a non-negative Lebesgue measurable function on $\mbb R^n$. If 
$T$ is the $(1/2)$-neighborhood of a line $L$ in $\mbb R^n$, we say $T$ is a 
1-tube and call $L$ the core line of $T$. For such tubes, 
we define
\begin{displaymath}
w(T) = \int_T w(x) dx.
\end{displaymath}

If $\sigma$ is the surface measure on  $S$, then the Mizohata-Takeuchi 
conjecture (\cite{sm:cauchy}, \cite{jt:evolution}, and \cite{jt:kowalew}) 
asserts that
\begin{displaymath}
\int |Ef(x)|^2 w(x) dx \lct \big( \sup_T w(T) \big) \| f \|_{L^2(\sigma)}^2 
\end{displaymath}
for all $f \in L^2(\sigma)$, where the sup is taken over all 1-tubes $T$ in 
$\mbb R^n$. This conjecture is open in all dimensions $n \geq 2$, but there 
are some known partial results. For example, the conjecture is true if $S$ 
is the unit sphere $\mbb S^{n-1} \subset\mbb R^n$ and the weight $w$ is 
radial. This result was shown, independently, in \cite{brv:helmholtz} and 
\cite{cs:fourinv} (see also \cite{crs:radial} and \cite{csv:means}).  

There is also a known local version of the above estimate that was proved in
\cite{bbc:localized} for the unit circle $\mbb S^1$: to every $\epsilon > 0$ 
there is a constant $C_\epsilon$ such that
\begin{equation}
\label{localmt}
\int_{|x| \leq R} |Ef(x)|^2 w(x) dx \leq C_\epsilon R^\epsilon R^{1/2} 
\big( \sup_T w(T) \big) \| f \|_{L^2(\sigma)}^2 
\end{equation}
for all $f \in L^2(\sigma)$, where the sup is taken over all 1-tubes $T$ in 
$\mbb R^2$. (See (\ref{localmtinf}) in the next section for a slightly 
better estimate than (\ref{localmt}).)   

For more results on the Mizohata-Takeuchi conjecture and its applications in 
Fourier analysis and PDE, we refer the reader to the papers 
\cite{bcsv:stein}, \cite{bbcrv:solutions}, \cite{ac:occupancy}, 
\cite{bn:tomography}, and \cite{chv:mtkak}.

The main results of this paper are stated in the following two theorems. 
Both theorems are stated in dimension $n=2$, but their method of proof 
applies and may very well lead to new results in higher dimensions; we 
restrict ourselves to $\mbb R^2$ to keep the technical details from 
distracting our attention from the main ideas. In both theorems, we think of 
$\sigma$ as a finite Borel measure on $\mbb R^2$ rather than a measure on a 
curve $S$, consider that the extension operator $E$ is associated with 
$\sigma$ rather than $S$, and study the Mizohata-Takeuchi conjecture through 
the decay properties of the Fourier transform of $\sigma$. (When $\sigma$ is 
supported on $S$, the decay of $\widehat{\sigma}$ often reflects some 
geometric property of $S$.)  

For $m \in \mbb R$, we denote by $\mbb T_m$ the set of all 1-tubes in 
$\mbb R^2$ whose core lines are parallel to either of the two lines 
$\{ x=(x_1,x_2) \in \mbb R^2 : m x_1+x_2 = 0 \mbox{ or } x_1 + m x_2 =0 \}$.
We note that if $m \not= 0$, then $\mbb T_m = \mbb T_{1/m}$.

\begin{thm}
\label{mainjone}
Suppose $0 < \delta \leq 1$ and $\sigma$ is a finite Borel measure on 
$\mbb R^2$ whose Fourier transform obeys the decay estimate
\begin{equation}
\label{decayjone}
|\widehat{\sigma}(x_1,x_2)| \leq \left\{
\begin{array}{ll}
C |x_1|^{-\delta} & \mbox{ if $|x_1| \geq 1$ and $|x_2| \leq 1$,} \\
C |x_2|^{-\delta} & \mbox{ if $|x_2| \geq 1$ and $|x_1| \leq 1$,} \\
C |x_1 x_2|^{-\delta} & \mbox{ if $|x_1| \geq 1$ and $|x_2| \geq 1$.} 
\\
\end{array} \right.
\end{equation} 
Also, suppose that the weight $w$ is a tensor function of the form
\begin{equation}
\label{tensor}
w(x) = w(x_1,x_2) = \widetilde{w}(a x_1) \widetilde{w}(b x_2),	
\end{equation}
where $\widetilde{w} : \mbb R \to [0,\infty)$ is Lebesgue measurable and
$a,b \in \mbb R$ are positive constants. Then, for $q > 1/\delta$, we have
\begin{equation}
\label{starjone}
\int |Ef(x)|^q w(x) dx 
\lct \big( \sup_{T \in \mbb T_{a/b}} w(T) \big) \| f \|_{L^2(\sigma)}^q 	
\end{equation}
for all $f \in L^2(\sigma)$. {\rm (}The implicit constant in {\rm 
(\ref{starjone})} \!\!\! depends only on $C$, $\delta$, $q$, and 
$\| \sigma \|$.{\rm )}
\end{thm}

Condition (\ref{tensor}) on $w$ is perhaps the `opposite' of the radial 
condition of \cite{brv:helmholtz} and \cite{cs:fourinv}. However, when 
$\sigma$ is arc length measure on the unit circle $\mbb S^1$, the decay
exponent of $\widehat{\sigma}$ (as given in (\ref{decayjone})) will be 
$\delta=1/4$, and hence Theorem \ref{mainjone} will only give us 
(\ref{starjone}) for $q > 4$ (see also (\ref{globalmtinf}) in the next 
section), whereas in the radial case \cite{brv:helmholtz} and 
\cite{cs:fourinv} tell us that (\ref{starjone}) (with 
$\sup_{T \in \mbb T_{a/b}} w(T)$ replaced by $\sup_T w(T)$) is true for 
$q \geq 2$.  

For $v \in \mbb R^2$ with $|v|=1$, we denote by $\mbb T_v$ the set of all 
1-tubes in $\mbb R^2$ whose core lines are perpendicular to $v$. 

\begin{thm}
\label{mainjtwo}
Suppose $0 < \delta \leq 1$, $v \in \mbb R^2$, $|v|=1$, and $\sigma$ is a 
finite Borel measure on $\mbb R^2$ whose Fourier transform obeys the decay 
estimate
\begin{equation}
\label{decayjtwo}
|\widehat{\sigma}(x)| \leq \frac{C}{|x \cdot v|^\delta} 
\end{equation} 
if $|x| \geq 1$. Also, suppose that $w : \mbb R^2 \to [0,\infty)$ is
Lebesgue measurable. Then, for $q > 2/\delta$, we have
\begin{equation}
\label{starjtwo}
\int |Ef(x)|^q w(x) dx 
\lct \big( \sup_{T \in \mbb T_v} w(T) \big) \| f \|_{L^2(\sigma)}^q 	
\end{equation}
for all $f \in L^2(\sigma)$. {\rm(}The implicit constant in {\rm 
(\ref{starjtwo})} \!\!\! depends only on $C$, $\delta$, $q$, and 
$\| \sigma \|$.{\rm )}
\end{thm}

As a consequence of Theorem \ref{mainjtwo}, we obtain the following result
concerning the extension operator associated with convex curves in the 
plane.

\begin{coro}
\label{convex}
Suppose $\gamma : [0,c] \to \mbb R$ satisfies the following properties: 
$\gamma$ and $\gamma'$ are convex with $\gamma(0)=\gamma'(0)=0$, 
$\gamma \in C^2((0,c]) \cap C^3((0,c))$, $(\gamma''^{1/2}/\gamma')' \leq 0$,
and $\gamma(t) \gamma''(t) / \gamma'(t)^2 \leq C$ for all $0 < t \leq c$.
Let $\sigma$ be the measure on the curve $(t,\gamma(t))$ given by 
$d\sigma(t)= \gamma''(t)^{1/2} dt$. Then, for $q > 4$, we have
\begin{displaymath}
\int |Ef(x)|^q w(x) dx 
\lct \big( \sup_{T \in \mbb T_{(0,1)}} w(T) \big) \| f \|_{L^2(\sigma)}^2 
\end{displaymath}
for all $f \in L^2(\sigma)$ and Lebesgue measurable 
$w : \mbb R^2 \to [0,\infty)$.
\end{coro}

\begin{proof}
We know from \cite[Proposition 3.1]{cz:flat} that 
$|\widehat{\sigma}(x_1,x_2)| \lct |x_2|^{-1/2}$ for all $x_2$ (uniformly in 
$x_1$), so the corollary follows from applying Theorem \ref{mainjtwo} with 
$\delta=1/2$ and $v=(0,1)$. 
\end{proof}

We note that the convex curves $(t,e^{-1/t^m})$, $0 \leq t \leq c_m$, 
$m \in \mbb N$, which are exponentially flat at the origin, fall within the 
scope of Corollary \ref{convex}.

We also note that if the curve is completely flat and $\sigma$ is arc length
measure on it, then the estimate of Corollary \ref{convex} is true for all 
$q > 2$, but with $\mbb T_{(0,1)}$ changed appropriately: if $v$ is a unit 
vector that is parallel to the curve, then, for $q > 2$, we have
\begin{equation}
\label{starflat}
\int |Ef(x)|^q w(x) dx 
\lct \big( \sup_{T \in \mbb T_v} w(T) \big) \| f \|_{L^2(\sigma)}^q 	
\end{equation}
for all $f \in L^2(\sigma)$ and $w : \mbb R^2 \to [0,\infty)$. It is not 
hard to prove (\ref{starflat}) (even for $q \geq 2$ and for any compact flat 
hypersurface in $\mbb R^n$, $n \geq 2$) directly without needing any 
information about $\widehat{\sigma}$, but it can also be easily obtained 
from Theorem \ref{mainjtwo} as follows.
We can take $\gamma(t)=0$, $0\leq t \leq 1$, so that $S$ is the line segment
$[0,1] \times \{ 0 \} \subset \mbb R^2$. Then 
\begin{displaymath}
|\widehat{\sigma}(x_1,x_2)| 
= \Big| \int_0^1 e^{-2 \pi i x_1 t} dt \Big| \lct \frac{1}{|x_1|} 
\end{displaymath}
for all $(x_1,x_2) \in \mbb R^2$, and (\ref{starflat}) follows by applying 
Theorem \ref{mainjtwo} with $v=(1,0)$ and $\delta = 1$.

So, while non-weighted restriction estimates (such as Stein-Tomas) require 
$\widehat{\sigma}$ to decay uniformly in all directions, Mizohata-Takeuchi 
estimates can hold by exploiting the decay of $\widehat{\sigma}$ in a single 
direction. (See also Subsection 2.2 below.)  

Theorems \ref{mainjone} and \ref{mainjtwo} were motivated by the author's 
recent paper \cite{pems201030}. This will be discussed in the next section.
The subsequent sections will then be dedicated to the proofs of Theorems 
\ref{mainjone} and \ref{mainjtwo} by following the strategy of
\cite{pems201030}.

\section{The motivation behind Theorems \ref{mainjone} and \ref{mainjtwo}}

Suppose $n \geq 1$ and $0 < \alpha \leq n$. Following \cite{plms12046} and 
\cite{pems201030}, for Lebesgue measurable functions 
$H: \mbb R^n \to [0,1]$, we define
\begin{displaymath}
A_\alpha(H)= \inf \Big\{ C : \int_{B(x_0,R)} H(x) dx \leq C R^\alpha
\mbox{ for all } x_0 \in \mbb R^n \mbox{ and } R \geq 1 \Big\},
\end{displaymath}
where $B(x_0,R)$ denotes the ball in $\mbb R^n$ of center $x_0$ and radius 
$R$. We say $H$ is a {\it weight of fractal dimension} $\alpha$ if 
$A_\alpha(H) < \infty$. We note that $A_\beta(H) \leq A_\alpha(H)$ if 
$\beta \geq \alpha$, so the phrase ``$H$ is a weight of fractal dimension 
$\alpha$'' is not meant to assign a dimension to the function $H$, rather, 
it is just another way for saying that $A_\alpha(H) < \infty$.

Suppose $n \geq 2$, the surface $S \subset \mbb R^n$ has a nowhere vanishing 
Gaussian curvature, and $\sigma$ is surface measure on $S$. In 
\cite{pems201030} (see the $\alpha = n/2$ case of 
\cite[Theorem 2.1]{pems201030}), it was proved that the restriction 
estimate\footnote{In \cite[Theorem 2.1]{pems201030}, $S$ was assumed to have 
a strictly positive second fundamental form, but the third paragraph 
following the statement of \cite[Theorem 2.1]{pems201030} explains that it 
suffices to assume that $S$ has a nowhere vanishing Gaussian curvature when 
$\alpha \leq n/2$.}
\begin{equation}
\label{knownweighted}
\int |Ef(x)|^q H(x) dx \lct A_{n/2}(H) \| f \|_{L^2(\sigma)}^q 
\end{equation}
holds for all functions $f \in L^2(\sigma)$ and weights $H$ on $\mbb R^n$ of 
dimension $n/2$ whenever $q > 2n/(n-1)$.

In dimension $n=2$, and for exponents $q > 4$, (\ref{knownweighted}) implies
that
\begin{equation}
\label{globalmtinf}
\int |Ef(x)|^q w(x) dx \lct \inf_{v \in \mbb S^1}
\big( \sup_{T \in \mbb T_v} w(T) \big) \| f \|_{L^2(\sigma)}^q 
\end{equation}
for all $f \in L^2(\sigma)$, as we shall see in the next subsection.

By H\"{o}lder's inequality and the inequality (\ref{forcs}) of the next
subsection, (\ref{globalmtinf}) gives the following improvement on 
(\ref{localmt}): 
\begin{equation}
\label{localmtinf}
\int_{|x| \leq R} |Ef(x)|^2 w(x) dx \leq C_\epsilon R^\epsilon R^{1/2} 
\inf_{v \in \mbb S^1} \big( \sup_{T \in \mbb T_v} w(T) \big) 
\| f \|_{L^2(\sigma)}^2.
\end{equation}

\subsection{Proof that (\ref{knownweighted}) implies (\ref{globalmtinf})} 

We may assume that $\| w \|_{L^\infty} \leq 1$. Let $v \in \mbb S^1$. If 
$B_R$ is a ball in $\mbb R^2$ of radius $R \geq 1$, then, covering $B_R$ by 
tubes $T_1, \ldots, T_N \in \mbb T_v$ having disjoint interiors, we see that  
\begin{displaymath}
\int_{B_R} w(x) dx = \sum_{l=1}^N \int_{B_R \cap T_l} w(x) dx
\leq N \sup_{T_l} w(T_l) \lct \big( \sup_{T \in T_v} w(T) \big) R.
\end{displaymath}
Thus $w$ is a weight on $\mbb R^2$ of fractal dimension $\alpha=1$, and
\begin{displaymath}
A_1(w) \lct \sup_{T \in \mbb T_v} w(T).
\end{displaymath}
Since this is true for every unit vector $v$, it follows that
\begin{equation}
\label{forcs}
A_1(w) \lct \inf_{v \in \mbb S^1} \big( \sup_{T \in \mbb T_v} w(T) \big).
\end{equation}
Applying (\ref{knownweighted}) with $H=w$ and $n=2$, we obtain 
(\ref{globalmtinf}). 

\subsection{Proof of (\ref{globalmtinf}) via Theorem \ref{mainjtwo}}

Since everything in this section is done under the assumption that $S$ has a
nowhere vanishing Gaussian curvature, we have the following decay estimate 
on the Fourier transform of the arc length measure $\sigma$ on $S$:
\begin{displaymath}
|\widehat{\sigma}(x)| \lct \frac{1}{|x|^{1/2}} 
\lct \frac{1}{|x \cdot v|^{1/2}}
\end{displaymath}
for all $x \in \mbb R^2$ and $v \in \mbb S^1$. Therefore, we can apply 
Theorem \ref{mainjtwo} with $\delta = 1/2$ to get (\ref{starjtwo}) for all 
$q > 4$ and $v \in \mbb S^1$, and (\ref{globalmtinf}) immediately follows.
 
\subsection{Strategy of \cite{pems201030}}

Most of the recent progress on the restriction problem in harmonic analysis 
relied on Guth's polynomial partitioning method that he introduced in
\cite{guth:poly}, and which was subsequently used in several papers (see
\cite{plms12046}, \cite{guth:poly2}, \cite{ghi:variable}, \cite{hr:highdim}, 
\cite{dgowwz:falconer}, \cite{ow:cone}, \cite{42over13}, and 
\cite{bs:improved}) to obtain new results on the restriction problem in both 
the weighted and non-weighted settings.

Guth's polynomial partitioning method upgrades restriction estimates from 
low `algebraic dimensions' to higher ones. The main innovation of 
\cite{pems201030} is a method for upgrading restriction estimates from low 
fractal dimensions to higher ones. 

For example, due to the decay estimate 
$|\widehat{\sigma}(x)| \lct |x|^{-(n-1)/2}$, (\ref{knownweighted}) is easy 
to prove in low fractal dimensions (more precisely, in the regime 
$0 < \alpha < (n-1)/2$). Since we are interested in the estimate at dimension 
$\alpha=n/2$, it becomes crucial to have a mechanism that will allow us to 
upgrade restriction estimates from low fractal dimensions to higher ones.  

This was done in \cite{pems201030} through a certain dimensional 
maximal-function that was defined as
\begin{displaymath}
M F(\alpha) = 
\Big( \sup_H \frac{1}{A_\alpha(H)} \int F(x)^\alpha H(x)dx \Big)^{1/\alpha},
\end{displaymath}
where $H$ ranges over all weights on $\mbb R^n$ that obey the condition 
$0 < A_\alpha(H) < \infty$, and which was shown to satisfy the following 
inequality.

\begin{alphthm}[{\cite[Theorem 1.1]{pems201030}}]
\label{dimineqzero}
Suppose $n \geq 1$ and $0 < \beta < \alpha \leq n$. Then
\begin{displaymath}
M F(\alpha) \leq M F(\beta)
\end{displaymath}
for all non-negative Lebesgue measurable functions $F$ on $\mbb R^n$.
\end{alphthm}
 
To prove Theorem \ref{mainjone} of this paper, we first define a suitable
notion of fractal dimension that will allow us to use the decay estimate
(\ref{decayjone}) to establish a low-dimensional version of 
(\ref{starjone}). We then define a suitable dimensional maximal function and 
prove an inequality (as in Theorem \ref{dimineqzero}) that will allow us to 
upgrade the low-dimensional version of (\ref{starjone}) to (\ref{starjone})
itself.

We follow the same strategy to prove Theorem \ref{mainjtwo}, which also 
needs its own notion of fractal dimension and dimensional maximal function.

\section{Proof of Theorem \ref{mainjone}}

\subsection{The dimensional maximal function}

Given a point $x=(x_1,x_2) \in \mbb R^2$ and numbers $R_1, R_2 \geq 1$, we 
define the set $\mbb B(x,R_1,R_2)$ to be the following subset of $\mbb R^4$:
\begin{displaymath}
\{ (y,z)=(y_1,y_2,z_1,z_2) \in \mbb R^4 : |y_1+z_1-x_1| \leq R_1 
   \mbox{ and } |y_2+z_2-x_2| \leq R_2 \}.
\end{displaymath}

Suppose $0 < \alpha \leq 2$ and $H : \mbb R^2 \to [0,1]$ is a Lebesgue 
measurable function. We define $\mbb A_\alpha(H)$ to be the {\em square 
root} of
\begin{displaymath}
\inf \Big\{ C : \int_{\mbb B(x,R_1,R_2)} H(y) H(z) d(y,z) 
                \leq C (R_1 R_2)^\alpha 
	    \mbox{ for all } x \in \mbb R^2 \mbox{ and } R_1, R_2 \geq 1 \Big\},
\end{displaymath}
where $d(y,z)$ is Lebesgue measure on $\mbb R^4$. We say $H$ is a {\it 
weight of fractal dimension} $\alpha$ if $\mbb A_\alpha(H) < \infty$. Same 
as with the definition of $A_\alpha(H)$ in the previous section, here, we 
also have $\mbb A_\beta(H) \leq \mbb A_\alpha(H)$ if $\beta \geq \alpha$, 
and the phrase ``$H$ is a weight of fractal dimension $\alpha$'' only means 
that $\mbb A_\alpha(H) < \infty$.

Also, for Lebesgue measurable $F : \mbb R^2 \to [0, \infty)$, we define
\begin{displaymath}
\mbb M F(\alpha) = \Big( \sup \frac{1}{\mbb A_\alpha(H)} 
                         \int F(x)^\alpha H(x) dx \Big)^{1/\alpha},
\end{displaymath}
where the sup is taken over all $H$ that obey the condition 
$0 < \mbb A_\alpha(H) < \infty$.

The following theorem and its corollary were motivated by 
\cite[Theorem 1.1]{pems201030} and \cite[Corollary 2.1]{pems201030}.

\begin{thm}
\label{maxineqone}
Suppose $0 < \beta < \alpha \leq 2$. Then
\begin{displaymath}
\mbb M F(\alpha) \leq \mbb M F(\beta)
\end{displaymath}
for all non-negative Lebesgue measurable functions $F$ on $\mbb R^2$.
\end{thm}

\begin{proof}
Let $F$ be a non-negative Lebesgue measurable functions $F$ on $\mbb R^2$.
For $N \in \mbb N$, we let $\chi_N$ be the characteristic function of the 
set
\begin{displaymath}
B(0,N) \cap \{ x \in \mbb R^2 : F(x) \leq N \}
\end{displaymath}
and $F_N= N^{-1} \chi_N F$. We also let $I = \int F_N(x) H(x) dx$. Since the 
tensor function $F_N(y) F_N(z)$ on $\mbb R^4$ is bounded by 1 and supported 
in the set
\begin{displaymath}
\{ (y,z)=(y_1,y_2,z_1,z_2) \in \mbb R^4 : |y_1|, |y_2|, |z_1|, |z_2| \leq N 
\},
\end{displaymath}
which is in turn contained in $\mbb B(0,2N,2N)$, it follows by the 
Fubini-Tonelli theorem that
\begin{eqnarray*}
I^2 
&  =   & \int F_N(y) F_N(z) H(y) H(z) d(y,z) \\
& \leq & \int_{\mbb B(0,2N,2N)} H(y) H(z) d(y,z) \\
& \leq & \mbb A_\alpha(H)^2 (2N)^{2\alpha},
\end{eqnarray*}
so that $I \leq 2^\alpha N^\alpha \mbb A_\alpha(H)$. Letting $\beta_0=1$ and 
$C_0= 2^\alpha N^\alpha$, we get
\begin{equation}
\label{reversezero}
\int F_N(x)^{\beta_0} H(x) dx \leq C_0 \mbb A_\alpha(H).
\end{equation}

For $x \in \mbb R^2$ and $R_1, R_2 \geq 1$, (\ref{reversezero}) and 
H\"{o}lder's inequality (applied with respect to the measure 
$H(y) H(z) d(y,z)$) tell us that
\begin{eqnarray*}
\lefteqn{\int_{\mbb B(x,R_1,R_2)} F_N(y)^{\beta_0/p} F_N(z)^{\beta_0/p} 
         H(y) H(z) d(y,z)} \\
& \leq & \Big( \int_{\mbb B(x,R_1,R_2)} F_N(y)^{\beta_0} F_N(z)^{\beta_0} 
               H(y) H(z) d(y,z) \Big)^{1/p} \\
&      & \times \Big( \int_{\mbb B(x,R_1,R_2)} H(y) H(z) d(y,z) \Big)^{1/p'} 
         \\
& \leq & \Big( \int F_N(x)^{\beta_0} H(x) dx \Big)^{2/p}  
         \Big( \mbb A_\alpha(H)^2 (R_1 R_2)^\alpha \Big)^{1/p'} \\
& \leq & C_0^{2/p} \mbb A_\alpha(H)^{(2/p)+(2/p')} (R_1 R_2)^{\alpha/p'} \\
&  =   & C_0^{2/p} \mbb A_\alpha(H)^2 (R_1 R_2)^{\alpha/p'},
\end{eqnarray*}
whenever $p, p' > 1$ are conjugate exponents. Choosing 
$p=\alpha/(\alpha-\beta)$, so that $\alpha / p' = \beta$, we conclude that
the function ${\mathcal H} : \mbb R^2 \to [0,1]$ defined by 
${\mathcal H}= F_N^{\beta_0/p} H$ is a weight on $\mbb R^2$ of fractal 
dimension $\beta$ with
\begin{displaymath}
\mbb A_\beta({\mathcal H}) \leq C_0^{1/p} \mbb A_\alpha(H).
\end{displaymath}
Therefore,
\begin{eqnarray*}
\int F_N(x)^\beta F_N(x)^{\beta_0/p} H(x) dx 
&   =  & \int F_N(x)^\beta {\mathcal H}(x) dx \\
& \leq & (\mbb M F_N(\beta))^\beta \mbb A_\beta({\mathcal H}) \\
& \leq & N^{-\beta} (\mbb M F(\beta))^\beta C_0^{1/p} \mbb A_\alpha(H),
\end{eqnarray*}
where we have used the fact that $F_N \leq N^{-1} F$ to conclude that
$\mbb M F_N(\beta) \leq N^{-1} \mbb M F(\beta)$. Letting
$M= N^{-\beta} \big( \mbb M F(\beta) \big)^\beta$, 
$\beta_1= \beta + (\beta_0/p)$, and $C_1 = M C_0^{1/p}$, we therefore have
\begin{equation}
\label{reverseone}
\int F_N(x)^{\beta_1} H(x) dx \leq C_1 \mbb A_\alpha(H).
\end{equation}

Now, letting $\beta_2 = \beta + (\beta_1/p)$ and $C_2 = M C_1^{1/p}$, and 
repeating the same procedure starting with (\ref{reverseone}) instead of 
(\ref{reversezero}), we arrive at
\begin{displaymath}
\int F_N(x)^{\beta_2} H(x) dx \leq C_2 \mbb A_\alpha(H).
\end{displaymath}
Therefore, proceeding in this fashion and using mathematical induction, we
obtain sequences $\{ \beta_k \}$ and $\{ C_k \}$, $k \geq 1$, so that
$\beta_k = \beta + (\beta_{k-1}/p)$ and $C_k = M C_{k-1}^{1/p}$ and
\begin{equation}
\label{reversekay}
\int F_N(x)^{\beta_k} H(x) dx \leq C_k \mbb A_\alpha(H).
\end{equation}

A simple calculation (see \cite[\S 5]{pems201030} for more details) reveals
that
\begin{displaymath}
\beta_k = \beta \frac{1-(1/p)^k}{1-(1/p)} + \frac{\beta_0}{p^k}
\hspace{0.25in} \mbox{ and } \hspace{0.25in}
C_k = M^{(1-(1/p)^k)/(1-(1/p))} C_0^{(1/p)^k},
\end{displaymath}
so that 
\begin{displaymath}
\lim_{k \to \infty} \beta_k = \frac{\beta}{1-(1/p)} = \alpha
\hspace{0.25in} \mbox{ and } \hspace{0.25in}
\lim_{k \to \infty} C_k = M^{\alpha/\beta} 
                        = N^{-\alpha} \big( \mbb M F(\beta) \big)^\alpha
\end{displaymath}
(recall that $p= \alpha / (\alpha - \beta)$).

Therefore, letting $k \to \infty$ in (\ref{reversekay}) and using Fatou's 
lemma, we see that
\begin{displaymath}
\int F_N(x)^\alpha H(x) dx 
\leq N^{-\alpha} \big( \mbb M F(\beta) \big)^\alpha.
\end{displaymath}
Recalling that $F_N= N^{-1} \chi_N F$, this becomes
\begin{displaymath}
\int \chi_N(x) F(x)^\alpha H(x) dx \leq \big( \mbb M F(\beta) \big)^\alpha.
\end{displaymath}

Therefore, letting $N \to \infty$, we see that
\begin{displaymath}
\int F(x)^\alpha H(x) dx \leq \big( \mbb M F(\beta) \big)^\alpha.
\end{displaymath}
Since this is true for all weights $H$ of dimension $\alpha$, we arrive at
our desired inequality $\mbb M F(\alpha) \leq \mbb M F(\beta)$.
\end{proof}

Returning to Fourier restriction, for $0 < \alpha \leq 2$ and 
$1 \leq p \leq \infty$, we define $Q(\alpha,p)$ to be the infimum of all
exponents $q > 0$ so that the following holds: there is a constant $C$ 
such
\begin{displaymath}
\int |Ef(x)|^q H(x) dx \leq C \mbb A_\alpha(H) \| f \|_{L^p(\sigma)}^q
\end{displaymath}
for all functions $f \in L^p(\sigma)$ and weights $H$ of fractal dimension 
$\alpha$.

\begin{coro}
\label{corotomaxineqone}
Suppose $0 < \beta < \alpha \leq 2$. Then
\begin{displaymath}
\frac{Q(\alpha,p)}{\alpha} \leq \frac{Q(\beta,p)}{\beta}.
\end{displaymath}
\end{coro}

\begin{proof}
The proof of this corollary is the same as the proof of
\cite[Corollary 2.1]{pems201030}. We include it here for the reader's 
convenience.

By the definition of $Q(\beta,p)$, to every $q > Q(\beta,p)$ there is a 
constant $C_q$ such that
\begin{displaymath}
\int |Ef(x)|^q {\mathcal H}(x) dx 
\leq C_q \mbb A_\beta({\mathcal H}) \| f \|_{L^p(\sigma)}^{q}
\end{displaymath}
for all functions $f \in L^p(\sigma)$ and weights ${\mathcal H}$ of fractal
dimension $\beta$. Letting $F= |Ef|^{q/\beta}$, we see that
\begin{displaymath}
\mbb M F(\beta) \leq \big( C_q \| f \|_{L^p(\sigma)}^{q} \big)^{1/\beta}.
\end{displaymath}
Applying Theorem \ref{maxineqone}, we get
\begin{displaymath}
\mbb M F(\alpha) \leq \big( C_q \| f \|_{L^p(\sigma)}^{q} \big)^{1/\beta}.
\end{displaymath}
Therefore,
\begin{displaymath}
\Big( \frac{1}{\mbb A_\alpha(H)} \int |Ef(x)|^{(\alpha/\beta)q} H(x) dx 
\Big)^{1/\alpha}
\leq \Big( C_q \| f \|_{L^p(\sigma)}^{q} \Big)^{1/\beta}
\end{displaymath}
for all functions $f \in L^p(\sigma)$ and weights $H$ of fractal dimension
$\alpha$. 

The definition of $Q(\alpha,p)$ now tells us that 
$(\alpha/\beta)q \geq Q(\alpha,p)$. Therefore,
\begin{displaymath}
q \geq \frac{\beta}{\alpha} Q(\alpha,p)
\end{displaymath}
for every $q > Q(\beta,p)$. Therefore, 
$Q(\beta,p) \geq (\beta / \alpha) Q(\alpha,p)$.
\end{proof}

\subsection{The low-dimensional estimate}

The availability of favorable restriction estimates in low dimensions hangs
on the fact that the Fourier transform of $\sigma$ satisfies the decay
condition (\ref{decayjone}), as we shall see during the proof of the 
following proposition.

\begin{prop}
\label{lowdimestone}
Suppose $0 < \alpha < \delta$. Then, for $q > 1$, we have
\begin{displaymath}
\int |Ef(x)|^q H(x) dx \lct \mbb A_\alpha(H) \| f \|_{L^2(\sigma)}^q
\end{displaymath}
for all functions $f \in L^2(\sigma)$ and weights $H$ of fractal dimension 
$\alpha$.
\end{prop}

\begin{proof}
We may assume that $H \in L^1(\mbb R^2)$. (Otherwise, we work with the 
weight $\chi_{B(0,R)} H$, obtain an estimate that is uniform in $R$, and 
then send $R$ to infinity using the fact that
$\mbb A_\alpha(\chi_{B(0,R)} H) \leq \mbb A_\alpha(H)$.) We define the 
measure $\mu$ on $\mbb R^2$ by $d\mu= H dx$, and observe that
\begin{equation}
\label{dimofmuone}
(\mu \times \mu)(\mbb B(x_0,R_1,R_2)) 
\leq \mbb A_\alpha(H)^2 (R_1 R_2)^\alpha
\end{equation}
for all $x_0 \in \mbb R^2$ and $R_1, R_2 \geq 1$. 

Let $f \in L^2(\sigma)$. We need to estimate $\| E f \|_{L^q(\mu)}$. We have
\begin{equation}
\label{bloc1}
\| E f \|_{L^q(\mu)}^q = \int_0^{\| f \|_{L^1(\sigma)}} q \lambda^{q-1} 
  \mu \big( \big\{ |E f| \geq \lambda \big\} \big) d\lambda.
\end{equation}
Also, the set $\{ |E f| \geq \lambda \}$ is contained in
\begin{displaymath}
\big\{ (\mbox{Re} \, E f)_+ \geq \frac{\lambda}{4} \big\} \cup 
\big\{ (\mbox{Re} \, E f)_- \geq \frac{\lambda}{4} \big\} \cup
\big\{ (\mbox{Im} \, E f)_+ \geq \frac{\lambda}{4} \big\} \cup 
\big\{ (\mbox{Im} \, E f)_- \geq \frac{\lambda}{4} \big\},
\end{displaymath}
where $(\mbox{Re} \, E f)_+$ and $(\mbox{Re} \, E f)_-$ are, respectively, 
the positive and negative parts of $\mbox{Re} \, Ef$; and similarly for 
$\mbox{Im} \, E f$. So, it is enough to estimate the $\mu$-measure of the 
set $\{ (\mbox{Re} \, E f)_+ \geq \lambda/4 \}$, which we denote by $G$.

Since $\lambda > 0$, the set where $(\mbox{Re} \, E f)_+ \geq \lambda/4$ is
the same as the set where $\mbox{Re} \, E f \geq \lambda/4$, so
\begin{displaymath}
\frac{\lambda}{4} \mu(G) 
\leq \int_G (\mbox{Re} \, E f) d\mu
= \mbox{Re} \int_G \, E f \, d\mu
= \mbox{Re} \int \chi_G \, \widehat{fd\sigma} \, d\mu
= \mbox{Re} \int \widehat{\chi_G d\mu} \, f d\sigma,
\end{displaymath}
and so (by Cauchy-Schwarz)
\begin{displaymath}
\lambda^2 \mu(G)^2 \leq 
16 \| f \|_{L^2(\sigma)}^2 \| \widehat{\chi_G d\mu} \|_{L^2(\sigma)}^2.
\end{displaymath}
Now
\begin{displaymath}
\| \widehat{\chi_G d\mu} \|_{L^2(\sigma)}^2
= \int \widehat{\chi_G d\mu} \; \overline{\widehat{\chi_G d\mu}} \, d\sigma
= \int \Big( \overline{\widehat{\chi_G d\mu}} \, d\sigma \widehat{\Big)} \, 
       \chi_G d\mu
= \int (\widehat{\sigma} \ast (\chi_G d\mu)) \, \chi_G d\mu,
\end{displaymath}
so
\begin{equation}
\label{bloc2}
\lambda^2 \mu(G)^2 \leq 16 \| f \|_{L^2(\sigma)}^2 
\int (\widehat{\sigma} \ast (\chi_G d\mu)) \, \chi_G d\mu.
\end{equation}

Let $\psi_0$ be an even $C_0^\infty$ function on $\mbb R$ satisfying 
$0 \leq \psi_0 \leq 1$, $\psi_0=1$ on the interval $[-1,1]$, and $\psi_0=0$ 
outside the interval $(-2,2)$. Also, for $l \in \mbb N$, define 
$\psi_l(r)=\psi_0(r/2^l)-\psi_0(r/2^{l-1})$. Then $\psi_l$ is supported in 
the set $2^{l-1} \leq |r| \leq 2^{l+1}$ for $l \geq 1$, and 
$\sum_{l=0}^\infty \psi_l =1$ on $\mbb R$. 

For $x = (x_1,x_2) \in \mbb R^2$, define 
$\Psi_{l,m}(x)=\psi_l(x_1) \psi_m(x_2)$. Then
\begin{displaymath}
\widehat{\sigma} \ast (\chi_G d\mu) = \sum_{l=0}^\infty \sum_{m=0}^\infty 
(\Psi_{l,m} \, \widehat{\sigma}) \ast (\chi_G d\mu).
\end{displaymath}
Since $\Psi_{0,0}$ is supported in $[-2,2] \times [-2,2]$, and $\Psi_{l,0}$ 
is supported in 
\begin{displaymath}
\{(x_1,x_2) \in \mbb R^2 : 2^{l-1} \leq |x_1| \leq 2^{l+1} \mbox{ and } 
                           |x_2| \leq 2 \}
\end{displaymath}
for $l \in \mbb N$, and $\Psi_{0,m}$ is supported in 
\begin{displaymath}
\{(x_1,x_2) \in \mbb R^2 : |x_1| \leq 2 \mbox{ and } 
                           2^{m-1} \leq |x_2| \leq 2^{m+1} \}
\end{displaymath}
for $m \in \mbb N$, and $\Psi_{l,m}$ is supported in the set
\begin{displaymath}
\{ (x_1,x_2) \in \mbb R^2 : 2^{l-1} \leq |x_1| \leq 2^{l+1} \mbox{ and }
                            2^{m-1} \leq |x_2| \leq 2^{m+1} \}
\end{displaymath}
for $(l,m) \in \mbb N^2$, it follows from (\ref{decayjone}) that 
$|\Psi_{l,m} \, \widehat{\sigma}| \lct 2^{-(l+m) \delta}$ for all 
$(l,m) \in \mbb N^2$. Of course, we also have 
$\| \widehat{\sigma} \|_{L^\infty} \lct 1$, so 
\begin{displaymath}
|\Psi_{l,m} \, \widehat{\sigma}| \lct 2^{-(l+m) \delta}
\hspace{0.25in} \mbox{ for all $l,m \geq 0$,}
\end{displaymath}
so
\begin{eqnarray*}
|(\Psi_{l,m} \, \widehat{\sigma}) \ast (\chi_G d\mu)(x)| 
& \leq & \int |\Psi_{l,m}(x-y) \, \widehat{\sigma}(x-y)| \chi_G(y) d\mu(y) 
         \\
& \lct & 2^{-(l+m) \delta} 
         \int_{|x_1-y_1| \leq 2^{l+1}, |x_2-y_2| \leq 2^{m+1}} 
         \chi_G(y) d\mu(y),
\end{eqnarray*}
and so
\begin{eqnarray*}
\lefteqn{\int |(\Psi_{l,m} \, \widehat{\sigma}) \ast (\chi_G d\mu)(x)| 
         \chi_G(x) d\mu(x)} \\
& \lct & 2^{-(l+m) \delta} \int_{\mbb B(0,2^{l+1},2^{m+1})} \chi_G(x) 
         \chi_G(y) d\mu(x) d\mu(y) \\
& \leq & 2^{-(l+m) \delta} (\mu \times \mu)(\mbb B(0,2^{l+1},2^{m+1}))
\end{eqnarray*}
for all $l,m \geq 0$. Invoking (\ref{dimofmuone}), this becomes
\begin{displaymath}
\int |(\Psi_{l,m} \, \widehat{\sigma}) \ast (\chi_G d\mu)| \chi_G d\mu \lct 
\mbb A_\alpha(H)^2 \, 2^{-(l+m)(\delta-\alpha)}.
\end{displaymath}
Therefore,
\begin{displaymath}
\int (\widehat{\sigma} \ast (\chi_G d\mu)) \chi_G d\mu \lct A_\alpha(H)^2
\sum_{l=0}^\infty \sum_{m=0}^\infty 2^{-(l+m)(\delta-\alpha)} 
\lct \mbb A_\alpha(H)^2,
\end{displaymath}
where we have used the assumption that $0 < \alpha < \delta$.

Returning to (\ref{bloc2}), we now have
$\lambda^2 \mu(G)^2 \lct \| f \|_{L^2(\sigma)}^2 \mbb A_\alpha(H)^2$, so
that 
\begin{displaymath}
\mu(G) \lct \| f \|_{L^2(\sigma)} \mbb A_\alpha(H) \lambda^{-1}.
\end{displaymath} 
Therefore, by (\ref{bloc1}),
\begin{displaymath}
\| E f \|_{L^q(\mu)}^q \lct  \mbb A_\alpha(H) \| f \|_{L^2(\sigma)} 
\int_0^{\| f \|_{L^1(\sigma)}}  \lambda^{q-2} d\lambda
\lct  \mbb A_\alpha(H) \| f \|_{L^2(\sigma)}^q
\end{displaymath}
provided $q > 1$.
\end{proof}

\subsection{The Mizohata-Takeuchi estimate (\ref{starjone})}

In the language of Corollary \ref{corotomaxineqone}, Proposition 
\ref{lowdimestone} says that $Q(\beta,2) \leq 1$ for all 
$0 < \beta < \delta$. So, applying Corollary \ref{corotomaxineqone} with 
$\alpha =1$, we get
\begin{displaymath}
\frac{Q(1,2)}{1} \leq \frac{Q(\beta,2)}{\beta} \leq \frac{1}{\beta}
\end{displaymath}
for all $0 < \beta < \delta$. Therefore, $Q(1,2) \leq 1 / \delta$. 
Therefore,
\begin{equation}
\label{highdimestone}
\int |Ef(x)|^q H(x) dx \lct \mbb A_1(H) \| f \|_{L^2(\sigma)}^q 
\end{equation}  
whenever $q > 1/\delta$, $H$ is a weight of fractal dimension one, and 
$f \in L^2(\sigma)$.

We are now ready to return to our Mizohata-Takeuchi estimate. We let $w$ be 
a non-negative function on $\mbb R^2$ and apply (\ref{highdimestone}) with 
$H_w=\| w \|_{L^\infty}^{-1} w$ to get
\begin{equation}
\label{mtverone}
\int |Ef(x)|^q w(x) dx 
\lct \| w \|_{L^\infty} \mbb A_1(H_w) \| f \|_{L^2(\sigma)}^q, 
\end{equation}
and it remains to estimate $\mbb A_1(H_w)$. We will be able to do this only
when $w$ is of the form (\ref{tensor}):
\begin{displaymath}
w(x) = w(x_1,x_2) = \widetilde{w}(ax_1) \widetilde{w}(bx_2)
\end{displaymath}
with $\widetilde{w} : \mbb R \to [0,\infty)$ and $a,b > 0$.

Recall that, for $m \in \mbb R$, $\mbb T_m$ was defined as follows: a 1-tube 
in $\mbb R^2$ belongs to $\mbb T_m$ if its core line is parallel to the line 
$\{ x=(x_1,x_2) \in \mbb R^2 : m x_1+x_2 = 0 \}$ or the line
$\{ x=(x_1,x_2) \in \mbb R^2 : x_1 + m x_2 =0 \}$.

For $x \in \mbb R^2$ and $R_1, R_2 \geq 1$, we have
\begin{displaymath}
\int_{\mbb B(x,R_1,R_2)} H_w(y)H_w(z) d(y,z) 
= \| w \|_{L^\infty}^{-2} J(a) J(b),
\end{displaymath}
where
\begin{displaymath}
J(a) = \int_{|y_1+z_1-x_1| \leq R_1} \widetilde{w}(ay_1) \widetilde{w}(az_1) 
       d(y_1,z_1)
\end{displaymath}
and
\begin{displaymath}
J(b) = \int_{|y_2+z_2-x_2| \leq R_2} 
       \widetilde{w}(by_2) \widetilde{w}(bz_2) d(y_2,z_2).
\end{displaymath}
Since both integrals $J(a)$ and $J(b)$ are the same, it is enough to bound
one of them.

We write 
\begin{displaymath}
J(b) = \int_{-\infty}^\infty \widetilde{w}(bz_2) I(z_2) dz_2
\hspace{0.25in} \mbox{ with } \hspace{0.25in}
I(z_2) = \int_{x_2-z_2-R_2}^{x_2-z_2+R_2} \widetilde{w}(by_2) dy_2.
\end{displaymath}
Applying the change of variables $v=(b/a) y_2$, the second integral becomes
\begin{displaymath}
I(z_2) = \frac{a}{b} \int_{(b/a)(x_2-z_2-R_2)}^{(b/a)(x_2-z_2+R_2)} 
         \widetilde{w}(av) dv.
\end{displaymath}
Plugging this back in $J(b)$, we see that
\begin{displaymath}
J(b) = \frac{a}{b} 
\int_{\widetilde{T}} \widetilde{w}(av) \widetilde{w}(bz_2) d(v,z_2)
= \frac{a}{b} \int_{\widetilde{T}} w(v,z_2) d(v,z_2)			 
\end{displaymath}
where $\widetilde{T}$ is the tube
\begin{displaymath}
\{ (v,z_2) \in \mbb R^2 : -(a/b) v + x_2 - R_2 \leq z_2 
                          \leq -(a/b) v + x_2 + R_2 \}.  
\end{displaymath}
Since $\widetilde{T}$ has cross-section $2 b R_2 / \sqrt{a^2+b^2}$, it 
follows that
\begin{displaymath}
J(b) \leq \frac{a}{b} \frac{2 b R_2}{\sqrt{a^2+b^2}} 
                \big( \sup_{T \in \mbb T_{a/b}} w(T) \big)
= \frac{2 a R_2}{\sqrt{a^2+b^2}} 
  \big( \sup_{T \in \mbb T_{a/b}} w(T) \big).		 
\end{displaymath} 
Therefore,
\begin{displaymath}
J(a) J(b) \leq \frac{2 b R_1}{\sqrt{a^2+b^2}} \frac{2 a R_2}{\sqrt{a^2+b^2}}
               \big( \sup_{T \in \mbb T_{a/b}} w(T) \big)^2
\leq 2 R_1 R_2 \big( \sup_{T \in \mbb T_{a/b}} w(T) \big)^2.
\end{displaymath}

We have proved that
\begin{displaymath}
\int_{\mbb B(x,R_1,R_2)} H_w(y)H_w(z) d(y,z)
\leq 2 \| w \|_{L^\infty}^{-2} \big( \sup_{T \in \mbb T_{a/b}} w(T) \big)^2
     R_1 R_2
\end{displaymath}
for all points $x \in \mbb R^2$ and numbers $R_1, R_2 \geq 1$. By the
definition of the functional $\mbb A_\alpha$, this implies that
\begin{displaymath}
\mbb A_1(H_w) \leq \sqrt{2} \, \| w \|_{L^\infty}^{-1}
                   \big( \sup_{T \in \mbb T_{a/b}} w(T) \big).
\end{displaymath}
Inserting this back in (\ref{mtverone}), we get
\begin{displaymath}
\int |Ef(x)|^q w(x) dx 
\lct \big( \sup_{T \in \mbb T_{a/b}} w(T) \big) \| f \|_{L^2(\sigma)}^q
\end{displaymath}
for $q > \delta^{-1}$, as promised.

\section{Proof of Theorem \ref{mainjtwo}}

\subsection{The dimensional maximal function}

Let $v$ be the unit vector given in the statement of Theroem \ref{mainjtwo}.
We start by introducing a notation. For $x \in \mbb R^2$ and $r > 0$, we let
$T(x,r)$ be the tube of cross-section $2r$ that contains the ball $B(x,r)$ 
and is parallel to $v$ (more precisely, the core line of $T(x,r)$ is 
parallel to the vector $v$).

Suppose $0 < \alpha \leq 2$ and $H : \mbb R^2 \to [0,1]$ is a Lebesgue 
measurable function. We define 
\begin{displaymath}
{\mathcal A}_\alpha(H) 
= \inf \Big\{ C : \int_{T(x_0,R)} H(x) dx \leq C R^\alpha 
	    \mbox{ for all } x_0 \in \mbb R^2 \mbox{ and } R \geq 1 \Big\}.
\end{displaymath}
We say $H$ is a {\it weight of fractal dimension} $\alpha$ if 
${\mathcal A}_\alpha(H) < \infty$. Also, for Lebesgue measurable 
$F : \mbb R^2 \to [0, \infty)$, we define
\begin{displaymath}
{\mathcal M} F(\alpha) = \Big( \sup \frac{1}{{\mathcal A}_\alpha(H)} 
                         \int F(x)^\alpha H(x) dx \Big)^{1/\alpha},
\end{displaymath}
where the sup is taken over all $H$ that obey the condition 
$0 < {\mathcal A_\alpha}(H) < \infty$.

\begin{thm}
\label{maxineqtwo}
Suppose $0 < \beta < \alpha \leq 2$. Then
\begin{displaymath}
{\mathcal M} F(\alpha) \leq {\mathcal M} F(\beta)
\end{displaymath}
for all non-negative Lebesgue measurable functions $F$ on $\mbb R^2$.
\end{thm}

The proof of Theorem \ref{maxineqtwo} is much closer to the proof of  
\cite[Theorem 1.1]{pems201030} than the proof of Theorem \ref{maxineqone} 
was, so we omit the details.

For $0 < \alpha \leq 2$ and $1 \leq p \leq \infty$, we define 
${\mathcal Q}(\alpha,p)$ to be the infimum of all numbers $q > 0$ so that 
the following holds: there is a constant $C$ such
\begin{displaymath}
\int |Ef(x)|^q H(x) dx \leq C {\mathcal A}_\alpha(H) \| f \|_{L^p(\sigma)}^q
\end{displaymath}
for all functions $f \in L^p(\sigma)$ and weights $H$ of fractal dimension 
$\alpha$. Theorem \ref{maxineqtwo} has the following corollary whose proof
is the same as that of Corollary \ref{corotomaxineqone} and  
\cite[Corollary 2.1]{pems201030}.

\begin{coro}
\label{corotomaxineqtwo}
Suppose $0 < \beta < \alpha \leq 2$. Then
\begin{displaymath}
\frac{{\mathcal Q}(\alpha,p)}{\alpha}  
\leq \frac{{\mathcal Q}(\beta,p)}{\beta}.
\end{displaymath}
\end{coro}

\subsection{The low-dimensional estimate}

In this subsection, we prove the following result, which, like Proposition
\ref{lowdimestone}, was motivated by \cite[Proposition 6.1]{pems201030}.

\begin{prop}
\label{lowdimesttwo}
Suppose $0 < \alpha < \delta$. Then, for $q > 2$, we have
\begin{displaymath}
\int |Ef(x)|^q H(x) dx \lct {\mathcal A}_\alpha(H) \| f \|_{L^2(\sigma)}^q
\end{displaymath}
for all functions $f \in L^2(\sigma)$ and weights $H$ of fractal dimension 
$\alpha$.
\end{prop}

\begin{proof}
We begin, as in the proof of Proposition \ref{lowdimestone}, by assuming 
that $H \in L^1(\mbb R^2)$, defining the measure $\mu$ on $\mbb R^2$ by 
$d\mu= H dx$, and observing that
\begin{equation}
\label{dimofmutwo}
\mu(T(x_0,R)) \leq {\mathcal A}_\alpha(H) R^\alpha
\end{equation}
for all $x_0 \in \mbb R^2$ and $R \geq 1$. 

Next, for $f \in L^2(\sigma)$, we write
\begin{equation}
\label{bloc1two}
\| E f \|_{L^q(\mu)}^q = \int_0^{\| f \|_{L^1(\sigma)}} q \lambda^{q-1} 
  \mu \big( \big\{ |E f| \geq \lambda \big\} \big) d\lambda,
\end{equation}
and observe that we will be done as soon as we derive a suitable bound on 
the $\mu$ measure of the set where $\mbox{Re} \, E f \geq \lambda/4$, which 
we again denote by the letter $G$. 

As in the proof of Proposition \ref{lowdimestone}, we have
\begin{equation}
\label{bloc2two}
\lambda^2 \mu(G)^2 \leq 16 \| f \|_{L^2(\sigma)}^2 
\int (\widehat{\sigma} \ast (\chi_G d\mu)) \, \chi_G d\mu.
\end{equation}
However, instead of showing that 
\begin{displaymath}
\int (\widehat{\sigma} \ast (\chi_G d\mu)) \, \chi_G d\mu 
\lct {\mathcal A}_\alpha(H)^2
\end{displaymath}
as we did in Proposition \ref{lowdimestone}, here, we will only be able to 
show that 
\begin{equation}
\label{lastbd}
|\widehat{\sigma} \ast (\chi_G d\mu)(x)| \lct {\mathcal A}_\alpha(H) 
\end{equation}
for all $x \in \mbb R^2$. 

We let $\psi_0$ be an even $C_0^\infty$ function on $\mbb R$ satisfying 
$0 \leq \psi_0 \leq 1$, $\psi_0=1$ on the interval $[-1,1]$, and $\psi_0=0$ 
outside the interval $(-2,2)$. For $l \in \mbb N$, we then define 
$\psi_l(r)=\psi_0(r/2^l)-\psi_0(r/2^{l-1})$. The function $\psi_l$ is 
supported in the set $2^{l-1} \leq |r| \leq 2^{l+1}$ for $l \geq 1$, and 
$\sum_{l=0}^\infty \psi_l =1$ everywhere on $\mbb R$. 

For $x = (x_1,x_2) \in \mbb R^2$, we also define 
$\Psi_l(x)=\psi_l(x \cdot v)$. Then
\begin{displaymath}
\widehat{\sigma} \ast (\chi_G d\mu) 
= \sum_{l=0}^\infty (\Psi_l \, \widehat{\sigma}) \ast (\chi_G d\mu).
\end{displaymath}
Since $\Psi_0$ is supported in the tube  $|x \cdot v| \leq 2$, and $\Psi_l$ 
is supported in the tube $2^{l-1} \leq |x \cdot v| \leq 2^{l+1}$ for 
$l \in \mbb N$, it follows from (\ref{decayjtwo}) that 
$|\Psi_l \, \widehat{\sigma}| \lct 2^{-l \delta}$ for all $l \in \mbb N$. 
Since $\| \widehat{\sigma} \|_{L^\infty} \lct 1$, this is also true for 
$l=0$. So,
\begin{displaymath}
|\Psi_l \, \widehat{\sigma}| \lct 2^{-l \delta}
\hspace{0.25in} \mbox{ for all $l \geq 0$,}
\end{displaymath}
and so
\begin{eqnarray*}
|(\Psi_l \, \widehat{\sigma}) \ast (\chi_G d\mu)(x)| 
& \leq & \int |\Psi_l(x-y) \, \widehat{\sigma}(x-y)| \chi_G(y) d\mu(y) 
         \\
& \lct & 2^{-l \delta} \int_{|(x-y) \cdot v| \leq 2^{l+1}} \chi_G(y) d\mu(y)
         \\
& \leq & 2^{-l \delta} \mu(T(x,2^{l+1})).
\end{eqnarray*}
Invoking (\ref{dimofmutwo}), this becomes
\begin{displaymath}
|(\Psi_l \, \widehat{\sigma}) \ast (\chi_G d\mu)(x)| 
\lct {\mathcal A}_\alpha(H) 2^{-l (\delta-\alpha)}
\end{displaymath}
for all $x \in \mbb R^2$ and $l \geq 0$. Summing over all $l$ using the 
assumption $\alpha < \delta$, this proves (\ref{lastbd}).

Inserting the bound (\ref{lastbd}) in (\ref{bloc2two}), we get
\begin{displaymath}
\lambda^2 \mu(G)^2 
\lct \| f \|_{L^2(\sigma)}^2 \int {\mathcal A}_\alpha(H) \chi_G d\mu
= \| f \|_{L^2(\sigma)}^2 {\mathcal A}_\alpha(H) \mu(G),
\end{displaymath}
which gives
\begin{displaymath}
\mu(G) \lct {\mathcal A}_\alpha(H) \| f \|_{L^2(\sigma)}^2 \lambda^{-2}.
\end{displaymath}
Therefore, by (\ref{bloc1two}),
\begin{displaymath}
\| E f \|_{L^q(\mu)}^q \lct  {\mathcal A}_\alpha(H) \| f \|_{L^2(\sigma)}^2 
\int_0^{\| f \|_{L^1(\sigma)}}  \lambda^{q-3} d\lambda
\lct  {\mathcal A}_\alpha(H) \| f \|_{L^2(\sigma)}^q
\end{displaymath}
provided $q > 2$.
\end{proof}

\subsection{The Mizohata-Takeuchi estimate (\ref{starjtwo})}

In the language of Corollary \ref{corotomaxineqtwo}, Proposition 
\ref{lowdimesttwo} says that ${\mathcal Q}(\beta,2) \leq 2$ for all 
$0 < \beta < \delta$. So, applying Corollary \ref{corotomaxineqtwo} with 
$\alpha =1$, we get
\begin{displaymath}
\frac{{\mathcal Q}(1,2)}{1} \leq \frac{{\mathcal Q}(\beta,2)}{\beta} 
\leq \frac{2}{\beta}
\end{displaymath}
for all $0 < \beta < \delta$. Therefore, $Q(1,2) \leq 2 / \delta$, which 
means that
\begin{equation}
\label{highdimesttwo}
\int |Ef(x)|^q H(x) dx \lct {\mathcal A}_1(H) \| f \|_{L^2(\sigma)}^q 
\end{equation}  
whenever $q > 2/\delta$, $H$ is a weight of fractal dimension one, and 
$f \in L^2(\sigma)$.

If the function $w$ is as given in the statement of Theorem \ref{mainjtwo},
then it is easy to see that
\begin{displaymath}
\int_{T(x_0,R)} w(x) dx \lct \big( \sup_{T \in \mbb T_v} w(T) \big) R
\end{displaymath} 
for all $x_0 \in \mbb R^2$ and $R \geq 1$. Thus
\begin{displaymath}
{\mathcal A}_1(\| w \|_{L^\infty}^{-1} w) 
\lct \| w \|_{L^\infty}^{-1} \big( \sup_{T \in \mbb T_v} w(T) \big).
\end{displaymath} 
Inserting this bound back in (\ref{highdimesttwo}) (with $H$ replaced by
$\| w \|_{L^\infty}^{-1} w$), we arrive at our Mizohata-Takeuchi estimate 
(\ref{starjtwo}).

\end{document}